\documentclass{amsart}

\usepackage{amscd}
\usepackage{amssymb}
\usepackage{stmaryrd}
\usepackage{color}

\newcommand{\e}{\varepsilon}
\newcommand{\eand}{\quad\text{and}\quad}
\newcommand{\R}{\mathbb{R}}

\renewcommand{\P}{\mathbb{P}}

\newcommand{\bkr}{\oblong} 
\newcommand{\bigbkr}{\bigbox} 

\usepackage{hyperref}
\hypersetup{
  colorlinks=true, 
   linkcolor=red,  
}
\usepackage[hyphenbreaks]{breakurl}

\theoremstyle{plain}
\newtheorem{thm}{Theorem}
\newtheorem{lem}[thm]{Lemma}
\newtheorem{cor}[thm]{Corollary}
\newtheorem{prop}[thm]{Proposition}

\theoremstyle{definition}

\newtheorem*{exercise*}{Exercise}
\newtheorem{eg}[thm]{Example}

\numberwithin{thm}{section}

\newcommand{\E}{\mathbb{E}}

\newcommand{\nhat}{\hat{n}}

\newcommand{\vs}[1]{\vec{#1}^*}

\newcommand{\vn}{\vec{n}}

\newcommand{\vns}{\vs{n}}

\newcommand{\vc}{\vec{c}}

\newcommand{\vw}{\vec{w}}

\newcommand{\eps}{\e}

\newcommand{\dt}{\mathrm{d}t}

\DeclareMathOperator{\Cyl}{Cyl}

\begin{document}

\title[Some people have all the luck]{Some people have all the luck}

\author{Richard Arratia}
\address{Arratia: Department of Mathematics, University of Southern California, Department of Mathematics, 3620 S. Vermont Ave., KAP 104, Los Angeles, CA 90089-2532}
\email{rarratia \text{at} usc.edu}

\author{Skip Garibaldi}
\address{Garibaldi: Institute for Pure and Applied Mathematics, UCLA, 460 Portola Plaza, Box 957121, Los Angeles, California 90095-7121}
\email{skip \text{at} member.ams.org}

\author{Lawrence Mower}
\address{Mower: Palm Beach Post, 2751 S Dixie Highway, West Palm Beach, FL 33405}

\author{Philip B.~Stark}
\address{Stark: Department of Statistics, \#3860, University of California, Berkeley, CA 94720-3860}


\begin{abstract}
We look at the Florida Lottery records of winners of prizes worth \$600 or more.
Some individuals claimed large numbers of prizes.
Were they lucky, or up to something? 
We distinguish the ``plausibly lucky'' from the ``implausibly lucky'' by solving optimization problems that take into account the particular games each gambler won, where  
plausibility is determined by finding the minimum expenditure so that
if every Florida resident spent that much, the chance that any of them would win as
often as the gambler did would still be less than one in a million.
Dealing with dependent bets relies on the BKR inequality; solving the optimization problem
numerically relies on the log-concavity of the regularized Beta function.
Subsequent investigation by law enforcement confirmed that the gamblers 
we identified as ``implausibly lucky'' were indeed behaving illegally. 
\end{abstract}

\maketitle

\setcounter{tocdepth}{1}
\tableofcontents

It is unusual to win a lottery prize worth \$600 or more.
No one we know has. 
But ten people have each won more than 80 such prizes in the Florida Lottery.  
This seems fishy.  
Someone might get lucky and win the Mega Millions jackpot (a 1-in-259~million chance) 
having bought just one ticket.  
But it's implausible that a gambler would win many unlikely
prizes without having bet very many times.

How many?
We pose an optimization problem whose answer gives a lower bound on any 
sensible estimate of an alleged gambler's spending:
over all possible combinations of Florida Lottery bets, 
what is the minimum amount spent so that, if \emph{every}
Florida resident spent that much, the
chance that \emph{any} of them would win so many times is still less than one in a million?
If that amount is implausibly large compared to that gambler's means, we have
statistical evidence that she is up to something.

Solving this optimization problem in practice hinges on two math facts:
\begin{itemize}
\item an inequality that lets us bound the probability of winning dependent
bets in some situations in which we do not know precisely which bets were made.
\item log-concavity of the regularized Beta function, which lets us show
that any local minimizer attains the global minimal value.
\end{itemize}
We conclude that 2 of the 10 suspicious gamblers could just be lucky.
The other 8 are chiseling or spending implausibly large sums on lottery tickets.  
These results were used by one of us (LM) to focus on-the-ground investigations 
and to support an expos\'e of lax security in the Florida lottery \cite{PBP}.
We describe what those investigations found, and the policy consequences
in Florida and other states.

\section{How long can a gambler gamble?}
Is there a non-negligible probability that a pathological gambler of moderate means 
could win many \$600+ prizes?  
If not, we are done: our suspicion of these 10~gamblers is justified.

So, suppose a gambler starts with a bankroll of $S_0$ and buys 
a single kind of lottery ticket over and over again.  
If he spends his initial bankroll and all his winnings, 
how much would he expect to spend in total and how many prizes would he expect to collect
before going broke?

Let the random variable $X$ denote the value of a ticket, payoff minus cost.  
We assume that 
\begin{equation} \label{ruin.ass1}
\E(X) < 0,
\end{equation}
because that is the situation in the games where our suspicious winners claimed prizes.  
(It does infrequently happen that lottery tickets can have positive expectation, see \cite{GroteMatheson} or \cite{Finding}.) 
 Assumption \eqref{ruin.ass1} and the Law of Large Numbers say that a gambler with a finite
 bankroll eventually
will run out of money, with probability~1.  
The question is: \emph{how fast?}

Write $c > 0$ for the cost of the ticket, so that 
\begin{equation} \label{ruin.ass2}
\P(X \ge -c) = 1 \eand \P(X = -c) \ne 0.
\end{equation}
To illustrate our assumptions and notation, let's look at a concrete example of a Florida game, Play 4.  
It is based on the \emph{numbers} or \emph{policy} game formerly offered by organized crime, described in \cite{Numbers} and \cite{Sellin}. 
Variations on it are offered in most states that have a lottery.

\begin{eg}[Florida's Play 4 game] \label{play4.eg}
Our ten gamblers claimed many prizes in Florida's Play~4 game,
although in 2012 it only accounted for about 6\% of the Florida Lottery's \$4.45 billion in sales.
Here are the rules, simplified in ways that don't change the probabilities.

The Lottery draws a 4-digit random number twice a day.
A gambler can bet on the next drawing by 
paying $c = \$1$ for a ticket, picking a 4-digit number, and choosing ``straight'' or ``box.''  

If the gambler bets ``straight,'' she wins \$5000 if her number matches
the next 4-digit number exactly (which has probability $p = 10^{-4}$).
She wins nothing otherwise.  
The expected value of a straight ticket is $\E(X) = \$5000 \times 10^{-4} - \$1 = -\$0.50$.

If a gambler bets ``box,'' she wins if  her number is a permutation of the digits in
the next 4-digit number the Lottery draws.
She wins nothing otherwise.  
The probability of winning this bet depends on the number of distinguishable
permutations of the digits the gambler selects.

For instance, if the gambler bets on 1112, there are 4 possible permutations,
1112, 1121, 1211, and 2111.
This bet is a ``4-way box.''
It wins \$1198 with probability $1/2500 = 4 \times 10^{-4}$, since 
4 of the 10,000 equally likely outcomes are permutations of those four digits.
If the gambler bets on 1122, there are 6~possible permutations of the digits;
this bet is called a ``6-way box.''
It wins \$800 with probability $6 \times 10^{-4}$.  (The 6-way box is relatively unpopular, accounting for less than 1\% of Play 4 tickets.)
Buying such a ticket has expected value $\E(X) \approx -\$0.52$.
Similarly, there are 12-way and 24-way boxes.
\end{eg}

Returning to the abstract setting, the gambler's bankroll after $t$ bets is 
\[
S_t := S_0 + X_1 + X_2 + \cdots + X_t,
\]
where $X_1, \ldots, X_t$ are i.i.d.\ random variables with the same distribution as 
$X$, and $X_i$ is the net payoff of the $i$-th ticket.  
The gambler can no longer afford to keep buying tickets after the $T$th one, 
where $T$ is the smallest 
$t \ge 0$ for which $S_t < c$.

\begin{prop} \label{ruin.2}
In the notation of the preceding paragraph, 
\[
\frac{S_0 - c}{|\E(X)|} < \E(T) \le \frac{S_0}{|\E(X)|},
\]
with equality on the right if $S_0$ and all possible values of $X$ are integer multiples of $c$.
\end{prop}

In most situations, $S_0$ is much larger than $c$, and the two bounds are almost identical.  
In expectation, the gambler spends a total of $c\E(T)$ on tickets, 
including all of his winnings, which amount to $c\E(T) - S_0$.

\begin{proof}
By the definition of $T$ and \eqref{ruin.ass2},
\begin{equation} \label{ruin.eq}
0 \le \E(S_T) < c
\end{equation}
with equality on the left in case $S_0$ and $X$ are integer multiples of $c$.
Now the crux is to relate $\E(T)$ to $\E(S_T)$.  
If $T$ were constant (instead of random),  then $T = \E T$ and we could simply write 
\begin{equation} \label{wald.eq}
\E(S_T) = \E(S_0 + \sum_{i = 1}^{\E T} X_i) = S_0 + \E(T)\, \E(X)
\end{equation}
and combining this with \eqref{ruin.eq} would give the claim.  
The key is that equation \eqref{wald.eq} holds even though $T$ is random --- this is Wald's Equation (see, e.g., \cite[\S5.4]{Durrett:EOSP}).
The essential property is that $T$ is a \emph{stopping time}, i.e., for every $k > 0$, whether or not one places a $k$-th bet is determined just from the outcomes of the first $k - 1$ bets.
\end{proof}

You might recognize that in this discussion that we are considering a version of the 
gambler's ruin problem but with an unfair bet and where the house has infinite money; 
for bounds on gambler's ruin without these hypotheses, see, e.g., \cite{Ethier}.

\subsection*{A ticket with just one prize} 
The proposition lets us address the question from the beginning of this section. 
Suppose a ticket pays $j$ with probability $p$ and nothing otherwise; 
the expected value of the ticket, $\E(X) = pj - c$, is negative; 
and $j$ is an integer multiple of $c$.  
If a gambler starts with a bankroll of $S_0$ and spends it all on tickets, successively using the winnings to buy more tickets, then by Proposition \ref{ruin.2} the gambler should expect to 
buy $\E(T) = S_0/(c - pj)$ tickets, which means winning 
\[
\frac{c\E(T) - S_0}{j} = \frac{pS_0}{c-pj}.
\]
prizes.

\begin{eg} \label{badger.eg}
How many prizes might a compulsive gambler of ``ordinary'' means claim?
Surely some gamblers have lost houses, 
so let us say he starts with a bankroll worth $S_0 = \$$175,000, an 
amount between the median list price and the median sale price of a house in 
Florida~\cite{Zillow}.
If he always buys Play~4 6-way box tickets and recycles his winnings to buy more tickets, the previous paragraph shows that he can expect to win about 
\[
pS_0/(c-pj) = 6 \times 17.5 / 0.52 \approx \text{202 times.}
\]
This is big enough to put him among 
the top handful of winners in the history of the Florida lottery.
\end{eg}

Hence, the number of wins alone does not give evidence that
a gambler cheated.  
We must take into account the particulars of the winning bets.

\section{A toy version of the problem}

From here on, a ``win'' means a win large enough to be recorded;
for Florida, the threshold is \$600.
Suppose for the moment that a gambler only buys one kind of 
lottery ticket, and that each ticket is for a different drawing,
so that wins are independent.
Suppose each ticket has probability
$p$ of winning.

A gambler who buys $n$ tickets spends $cn$ and, on average, wins $np$ times.
This is intuitively obvious, and follows formally by modeling a
lottery bet as a Bernoulli trial with probability $p$ of success:
in $n$ trials we expect $np$ successes.  

We don't know $n$, and the gambler is unlikely to tell us.
But based on the calculation in the preceding paragraph,
we might guess that a gambler who won $W$ times bought roughly $W/p$ tickets.
Indeed, an unbiased estimate for $n$ is
$\nhat := W/p$,
corresponding to the gambler spending $c\nhat$ on tickets.  
Since $p$ is very small, like $10^{-4}$, the number $\nhat$ is 
big---and so is the estimated amount spent, $c\nhat$.  
(Note that this estimate includes any winnings
``reinvested'' in more lottery tickets.)

A gambler confronted with $\nhat$
might quite reasonably object that she is just very lucky, 
and that the true number of tickets she bought, $n$, is much smaller.
Under the assumptions in this section, her tickets are 
i.i.d. (independent, identically distributed)
Bernoulli trials, and the number of wins $W$ has a binomial
distribution with parameters $n$ and $p$, which lets
us check the plausibility of her claim by considering
\begin{equation} \label{prob.bin}
   D(n; w, p) :=\fbox{\parbox{1.4in}{probability of at least $w$ wins with $n$ tickets}} = \sum_{k=w}^n \binom{n}{k} p^k (1-p)^{n-k}.
\end{equation}

Modeling a lottery bet as a Bernoulli trial is precisely correct in the case of games like Play 4.  But for scratcher games, there is a very large pool from which the gambler is sampling without replacement by buying tickets; as the pool is much larger than the values of $n$ that we will consider, the difference between drawing tickets with and without replacement is negligible.

\begin{eg}[Louis Johnson] \label{forward.eg}
Of the 10 people who had won more than 80 prizes each in the Florida Lottery, the second most-frequent prize claimant was Louis Johnson.  He 
claimed $W = 57$ \$5,000 prizes from 
straight Play~4 tickets (as well as many prizes in many other games that we ignore in this example).
We estimate that he bought $\nhat = W/p = 570,000$ tickets at a
cost of \$570,000.

What if he claimed to only have bought $n =$~175,000 tickets?
The probability of winning at least $57$ times with 175,000 tickets is
\[
D(175000; 57, 10^{-4}) \approx 6.3 \times 10^{-14}.
\]
For comparison, by one estimate there are about 400 billion stars in our galaxy \cite{Milky}.  Suppose there were a list of all those stars, and two people independently pick a star at random from that list. The chance they would pick the same star is minuscule, yet it is still 40 times greater than the probability we just calculated.
It is utterly implausible that a gambler wins 57~times by buying
175,000 or fewer tickets. 
\end{eg}

\section{What this has to do with Joe DiMaggio} \label{DiMaggio.sec}

The computation in Example~\ref{forward.eg} does not directly answer
whether Louis Johnson is lucky or up to something shady.
The most glaring problem is that we have calculated the probability 
that a \emph{particular} innocent gambler who buys \$175,000 of Play~4 tickets would win 
so many times.
The news media have publicized 
some lottery coincidences as astronomically unlikely, 
yet these coincidences have turned out to be relatively 
unsurprising given the enormous number of people playing the lottery; see, for example,
\cite[esp.~p.~859]{DiaconisMosteller} or \cite{Stefanski} and the references therein.

Among other things, we need to check whether so many people are playing 
Play~4 so frequently that it's reasonably likely at least one of them would 
win at least 57~times.
If so, Louis Johnson might be that person, just 
like with Mega Millions: no particular ticket has a big chance of winning, 
but if there are enough gamblers, there is a big chance \emph{someone} wins.

We take an approach similar to how baseball probability enthusiasts attempt to answer the question,
\emph{Precisely how amazing was Joe DiMaggio?}  
Joe DiMaggio is famous for having the longest hitting streak in baseball:
he hit in 56 consecutive games in 1941.  
(The  modern player with the second longest hitting streak is Pete Rose, who hit in 44 consecutive games in 1978.)
One way to frame the question is to consider the probability that a randomly 
selected player gets a hit in a game,
and then estimate the probability that there is at least one hitting streak at least 
56~games long in the entire 
history of baseball. 
If a streak of 56 or more games is likely, then the answer to the question is ``not so amazing'';
DiMaggio just happened to be the person who had the unsurprisingly long streak.
If it is very unlikely that there would be such a long streak, then the answer is: DiMaggio was truly amazing.  (The conclusions in DiMaggio's case have been equivocal, see the discussion in \cite[pp.~30--38]{ProbTales}.)

Let's apply this reasoning to Louis Johnson's 57 Play~4 wins (Example~\ref{forward.eg}).
Suppose that $N$ gamblers bought Play~4 tickets during the relevant time period, 
each of whom spent at most \$175,000.
Then an upper bound on the probability that at least one such gambler would win at least 57~times 
is the chance of at least one success in $N$ 
Bernoulli trials, each of which has probability no larger than
$p \approx 6.3 \times 10^{-14}$ of success.  
(Louis Johnson represents a success.)
The trials might not be independent, because different gamblers might bet on the
same numbers for the same game, but the chance that at least one of the $N$ 
gamblers wins at least 57~times is
at most $Np$ by the Bonferroni bound (for any set of events $A_1, \ldots, A_N$, 
$\P(\cup_{i=1}^N A_i) \le \sum_{i=1}^N \P(A_i)$).

What is $N$?  
Suppose it's the current population of Florida, approximately 19~million.
Then the chance at least one person would win at least 57~times is no larger than
$19 \times 10^{6} \times 6.3 \times 10^{-14} = 0.0000012$, just over one in a million.

This estimate is crude because the estimated number of gamblers is very rough and of 
course the estimate is not at all sharp (it gives a lot away in the direction of making the gambler look less suspicious) 
because most people spend nowhere near \$175,000 on the lottery.
We are giving even more away because Louis Johnson won many 
other bets (his total winnings are, of course, dwarfed by the expected cash outlay).
Considering all these factors, one might reasonably
conclude that either Louis Johnson has a source of hidden of money---perhaps 
he is a wealthy heir with a gambling problem---or he is up to something.

\begin{eg}[Louis Johnson 2] \label{reverse.eg}
In Example~\ref{forward.eg} we picked the 
\$175,000 spending level almost out of thin air, based on Florida house prices as in Example \ref{badger.eg}.
Instead of starting with a limit on spending and deducing the probability of a number of wins, 
let's start with a probability,  $\e = 5 \times 10^{-14}$, and infer the minimum spending
required to have at least that probability of so many wins.

If Johnson buys $n$ tickets, then he wins at least 57~times with probability 
$D(n; 57, 10^{-4})$.
We compute $n_0$, the smallest $n$ such that 
\[
   D(n; 57, 10^{-4}) \ge \e,
\]
which gives $n_0 =$~174,000.  
Using the Bonferroni bound again, 
we find that the probability, if \emph{everyone} in Florida spent \$174,000 on straight Play~4
tickets, the chance that \emph{any} of them would win 57 times or more 
is less than one in a million.
\end{eg}

\section{Multiple kinds of tickets} 
Real lottery gamblers tend to wager on a variety of games with different
odds of winning and different payoffs.
Suppose they place $b$ different kinds of bets.  
(It might feel more natural to say ``games,'' but 
a gambler could place several dependent bets on a single Play~4 drawing:
straight, several boxes, etc.)  

Number the bets $1, 2, \ldots, b$.
Bet $i$ costs $c_i$ dollars and has probability $p_i$ of winning.
The gambler won more than the threshold on 
bet $i$ $w_i$ times.
We don't know $n_i$, the number of times the gambler wagered on bet $i$.
If we did know the vector $\vn = (n_1, n_2, \ldots, n_b)$,
then we might be able to calculate the probability:
\begin{equation} \label{Pdef.1}
     P(\vn; \vec{w}, \vec{p}) := \left({\parbox{3in}{probability of winning at least $w_i$ times on bet $i$ with $n_i$ tickets, for all $i$}} \right).
\end{equation}
As in Example~\ref{reverse.eg}, we can find a lower bound on the
amount spent to attain $w_i$ wins on bet $i$, $i = 1, \ldots, b$,
by solving
\begin{equation} \label{premother}
\vc \cdot \vns = \min_{\vn} \vc \cdot \vn\quad \text{s.t.}\quad \ n_i \ge w_i \eand  P(\vn; \vec{w}, \vec{p}) \ge \eps.
\end{equation}
For a typical gambler that we study, this lower bound $\vc \cdot \vns$ will be in the millions of dollars.  Thinking back to the ``Joe DiMaggio'' justification for why \eqref{premother} is a lower bound, it is clear that not every resident of Florida would spend so much on lottery tickets, and our gut feeling is that a more refined justification would produce a larger lower bound for the amount spent.

But how can we find $P(\vn; \vec{w}, \vec{p})$?  
If the different bets were on independent events 
(say, each bet is a different kind of scratcher ticket),
then
\begin{equation} \label{Peq.1}
     P(\vn; \vec{w}, \vec{p}) = \prod_{i=1}^b \left({\parbox{2in}{probability of winning at least $w_i$ times on bet $i$ with $n_i$ tickets}} \right) 
     = \prod_{i=1}^b D(n_i; w_i, p_i). 
\end{equation}
But gamblers can make dependent bets, in which case \eqref{Peq.1} does not hold.  
Fortunately, it is possible to derive an upper bound for the typical case,
as we now show.

\section{No dependent wins is almost as good as independent bets}\label{sect bkr}

For most of the 10~gamblers, we did not observe wins on dependent bets, 
such as a win on a straight ticket and a win on a 4-way box ticket for the same Play~4 drawing.  
We seek to prove Proposition \ref{cond} (below), which says that 
if there were no wins on dependent bets, then
treating the bets as if they were independent gives
an overall upper bound on the probability $P$ in \eqref{Pdef.1}.

Abstractly, we envision a finite number $d$ of independent drawings, such
as a sequence of Play~4 drawings.
For each drawing $j$, $j = 1, \ldots, d$, the gambler may bet any amount
on any of $b$ different bets (such as 1234 straight, 1344 6-way box, etc.), whose outcomes---for drawing $j$---may be 
dependent, but whose outcomes
on different draws are independent.  
We write $p_i$ for the probability that a bet on $i$ wins in any particular drawing;
$p_i$ is the same for all drawings $j$.

For  $i=1, \ldots, b$ and $j = 1, \ldots, d$,  let  $n_{ij} \in \{0,1\}$ 
be the indicator that the gambler wagered on bet $i$ in drawing $j$,  
so that $i$th row sum, $n_i := \sum_j n_{ij}$, is the total number of bets on $i$.   
We call the entire system of bets $B$,  represented by the $b$-by-$d$ 
zero-one  matrix $B=[n_{ij}]$.  

\begin{prop} \label{cond}
Suppose that, for each $i$, a gambler wagers on bet $i$ in $n_i$ 
different drawings, as specified by $B$, above.  
Given the bets $B$, consider the events
\[
    W_i := (\text{gambler wins bet $i$ at least $w_i$ times with bets $B$}),
\]
and the event
\[ 
I:= (\text{in each drawing $j$, the gambler wins at most one bet}).
\] 
Then
\begin{equation} \label{cond.1}
   \P(I \cap W_1 \cap \cdots \cap W_b) \le \prod_{i=1}^b \P(W_i).   
\end{equation}
\end{prop}

In our case, $\P(W_i) = D(n_i; w_i, p_i)$, so we restate \eqref{cond.1} as:
\begin{equation}\label{cond.2}
\P(I \cap W_1 \cap \cdots \cap W_b) \le \prod_{i=1}^b D(n_i; w_i, p_i).
\end{equation}

Proposition~\ref{cond} is intuitively plausible: even though the bets are not independent,
the drawings are, and event 
$I$ guarantees that any single drawing helps at most one of the events 
$\{W_i\}$ to occur.  
We prove Proposition~\ref{cond} as a corollary of an extension of a celebrated 
 result, the BKR inequality,  named for 
van~den~Berg--Kesten--Reimer, conjectured in \cite{BK}, and proved in 
\cite{Reimer} and \cite{BFiebig} (or see \cite{CPS}).   
The remainder of this section provides the details.
The original BKR inequality is stated as Theorem~\ref{bkr thm}.  
We separate the purely set-theoretic aspects of the discussion, 
in Section~\ref{bkr.def} and \ref{sect bkr set},  
from the probabilistic aspects, in Section~\ref{sect bkr prob}.

\subsection{The BKR operation $\bkr$} \label{bkr.def}
 Let $S$ be an arbitrary set,
and write $S^d$ for the Cartesian product of $d$ copies of $S$.   
Since our application is probability, we call an 
element $\omega = (\omega_1, \ldots, \omega_d) \in S^d$ an \emph{outcome}, and we call
any $A \subseteq S^d$  an \emph{event}. 

For a subset $J \subseteq \{ 1, \ldots, d \}$ and an outcome $\omega \in S^d$, the 
\emph{$J$-cylinder of $\omega$}, denoted $\Cyl(J, \omega)$, 
is the collection of $\omega' \in S^d$ such that $\omega'_j = \omega_j$ for all $j \in J$.  
For events $A_1, A_2, \ldots, A_b$, let 
$A_1 \bkr A_2 \bkr \cdots \bkr A_b \subseteq S^d$ be the set of $\omega$ for which there 
exist pairwise disjoint $J_1, J_2, \ldots, J_b \subseteq \{ 1, \cdots d \}$ 
such that $\Cyl(J_i, \omega) \subseteq A_i$ for all $i$.  
The case $b = 2$, where one combines just two events, is the context for the original 
BKR inequality as in \cite[p.~564]{BK}; the operation with $b > 2$ is 
new and is the main study of this section.

Here is another definition of $\bkr$ that might be more transparent.
Given an event $A \subseteq S^d$ and a subset $J \subseteq \{ 1, \ldots, d \}$, define the event
\[
  [A]_J  := \{ \omega \in A \mid  \Cyl(J,\omega)\subseteq A \} = \bigcup\nolimits_{\{\omega \mid  \Cyl(J, \omega) \subseteq A\}} \Cyl(J, \omega).
\]
Informally, $[A]_J$ consists of the outcomes in $A$, such that  by looking only at the coordinates indexed by $J$, one can tell that $A$ must have occurred.  
Evidently, for $A,   B \subseteq S^d$, 
\begin{equation} \label{monotone}
    A \subseteq B \text{ implies } A_J \subseteq B_J \quad \text{and} \quad
     J \subseteq K \text{ implies } A_J \subseteq A_K.  
\end{equation}
The definition of $\bkr$ becomes:
\begin{equation}\label{def bkr many}
 \bigbkr_{1 \le i \le b} A_i   := \bigcup_{\text{pairwise disjoint $J_1,\ldots, J_b \subseteq \{ 1, \ldots, d \}$}} [A_1]_{J_1} \cap [A_2]_{J_2} \cap \cdots \cap [A_r]_{J_b}.
\end{equation}
We read the above definition as ``$\bigbkr_{1 \le i \le b} A_i$ is the event that all $b$ events occur,  with $b$ disjoint sets of reasons to simultaneously certify the $b$ events.''
Informally,  the outcome $\omega$, observed only on  the coordinate indices in $J_i$,  supplies the ``reason''  that we can certify that event $A_i$ occurs.

Our notation  $\bigbkr_{1 \le i \le b} A_i   \equiv    A_1 \bkr A_2 \bkr \cdots \bkr A_b$ is intentionally
analogous to the notations for set intersection,  
$\bigcap_{1 \le i \le b} A_i   \equiv    A_1 \cap A_2 \cap \cdots \cap A_b$, and set union, 
$\bigcup_{1 \le i \le b} A_i   \equiv    A_1 \cup A_2 \cup \cdots \cup A_b$.
The multi-input operator $\bigbkr$ is, like set intersection $\bigcap$ and set union $\bigcup$, fully commutative, i.e., unchanged by any re-ordering of the inputs. 
Unlike intersection and union, $\bkr$ is not associative, as we now show.

\begin{eg} \label{skip example}
Take $S=\{0,1\}$, $d=3$,  and 
\[
A =  (0,*,*) \cup (1,0,*),  \quad B = (0, *, *) \cup (1, 1, *), \quad
C = (*,0,1),
\]
where we write for example $(1, 0, *) =   \{ (1, 0, 0), (1, 0, 1) \}= \Cyl(\{1, 2 \}, (1, 0, s)) $ for $s = 0, 1$ and $(0,*,*)= \{ (0,0,0), (0,0,1), (0,1,0), (0,1,1) \}$.  Note that
$|A|=|B|=6$.
Then  $A \bkr B = (0,*, *)$, $(A \bkr B) \bkr C = \{ (0,0,1) \}$ --- using $J_1 = \{1\}$ and $J_2 = \{2,3\}$ in \eqref{def bkr many} --- but
$B \bkr C = \{ (0,0,1) \}$ and $A \bkr (B \bkr C) = \emptyset$.   Also, $A \bkr B \bkr C = \emptyset$.
\end{eg}

\subsection{The connection between lottery drawings and $\bkr$}  \label{bkr.connect}
Before continuing to discuss the BKR operation $\bkr$ in the abstract,
we consider what it means for lottery drawings.
We take $S = 2^b$ to encode the results of a single draw: an element 
$s \in S$ answers, for each of the $b$ bets, whether that bet wins or not.
The sample space for our probability model is $S^d$; the $j$-th coordinate $\omega_j$ 
reports the results of the $b$ bets on the $j$-th draw.

It is easy to see that, in the notation of Proposition~\ref{cond},
\begin{equation}\label{contained}
 \left(  I \cap W_1 \cap \cdots \cap W_b \right)   \subseteq \bigbkr_1^b W_i.
 \end{equation}
Indeed, given an outcome $\omega \in  I \cap W_1 \cap \cdots \cap W_b$, 
we can take, for $i=1$ to $b$, $J_i := \{j \mid \text{on draw $j$, 
bet $i$ wins \emph{and} $n_{ij}=1$} \}$. 
Since $\omega \in I$, the sets $J_1,\ldots,J_b$ are mutually disjoint;
and since $\omega \in W_i$, $|J_i| \ge n_i$.
Hence,
$\Cyl(J_i,\omega) \subseteq W_i$, and thus $\omega \in [W_i]_{J_i}$, for $i=1$ to $b$. 

\begin{eg}
The left hand side of \eqref{contained} can be a strict subset of the right hand side.
For example, with $b=2$ bets and $d = 2$ draws, suppose 
that $w_1 = w_2 = 1$ and the gambler lays both bets on both draws.  
The outcome where both bets win on both draws is not in the left side of 
\eqref{contained} but is in $W_1 \bkr W_2$.   

To write this example out fully, we think of the binary encoding, $S=\{0,1,2,3\}$
corresponding to $\{00,01,10,11\}$, so that, for example, $0\in S$ 
represents a draw where both bets lose, $1 \in S$ represents the outcome 01 
where the first bet loses and the second bet wins,  $2 \in S$ represents the outcome 10 where the first bet wins and the second bet loses,
and $3 \in S$ represents the outcome 11 where both bets win. 

The event $I$ is the set of $\omega= (\omega_1, \omega_2)$ 
for which no coordinate $\omega_j$ is equal to 3. 
The event $W_1$ is the set of $\omega$ such that at least one of the coordinates is
equal to 1 or 3, and the event $W_2$ is the set of  $\omega$ such that at least one of the coordinates is 
equal to 2 or 3. 
Certainly,
\[
I \cap W_1 \cap W_2 = \{ (1, 2), (2, 1) \},
\]
yet
\[
W_1 \bkr W_2 = \{ (1, 2), (2, 1), (1, 3), (2, 3), (3, 1), (3, 2), (3, 3) \}.
\]
\end{eg}

 \subsection{Set theoretic considerations related to the BKR inequality}\label{sect bkr set}
 It is obvious that,
for events $B_1, \ldots, B_r \subseteq S^d$ and $J \subseteq \{1,\cdots,d\}$, 
\begin{equation} \label{capcup}
   \left[\bigcap\nolimits_{1 \le i \le r} B_i\right]_J = \bigcap_{1 \le i \le r} [B_i]_J \quad \text{and} \quad
 \left[\bigcup\nolimits_{1 \le i \le r} B_i\right]_J \supseteq \bigcup_{1 \le i \le r} [B_i]_J
\end{equation}
For unions, the containment may be strict, as in Example \ref{skip example}, where $A \cup B = S^d$ hence $[A \cup B]_\emptyset = S^d$, whereas $[A]_\emptyset = [B]_\emptyset = \emptyset$.

 \begin{lem}[Composition of cylinder operators] \label{cylcomp}
For $A \subseteq S^d$ and $J,K \subseteq \{1,\cdots,d\}$, 
\[
   [[A]_J]_K = [A]_{J \cap K}.
\]
\end{lem}
\begin{proof}   
Suppose first that $\omega \in [[A]_J]_K$.  
That is, $\Cyl(K, \omega) \subseteq A_J$: if $\omega'' \in S^d$ agrees with $\omega$ on $K$, 
then $\Cyl(J, \omega'') \subseteq A$.  
We must show that $\omega$ is in $A_{J \cap K}$; i.e., if $\omega''$ is in $\Cyl(J \cap K, \omega)$, then $\omega''$ is in $A$.

Given $\omega'' \in \Cyl(J \cap K, \omega)$, 
pick $\omega'$ to agree with $\omega$ on $K$ and $\omega''$ on $S^d \setminus K$.  
Then $\omega'$ agrees with $\omega''$ on 
$(S^d \setminus K) \cup (J \cap K)$, so on $J$, i.e., $\omega'' \in A$, proving $\subseteq$.

We omit the proof of the containment $\supseteq$, which is easier.
\end{proof}

\begin{prop} \label{bkr.incl}
For $A_1, A_2, \ldots, A_b \subseteq S^d$, we have:
\[
\bigbkr_1^b A_i \subseteq (((\cdots ((A_1 \bkr A_2) \bkr A_3) \cdots \bkr A_b.
\]
\end{prop}

\begin{proof}
By induction, using \eqref{monotone},  it suffices to prove that 
\[
\left( \bigbkr_1^b A_i \right) \subseteq  \left( \bigbkr_1^{b-1} A_i \right) \bkr A_b.
\]
With unions over $K \subseteq \{ 1, \ldots, d \}$ and pairwise disjoint $J_1, J_2, \ldots$, 
\begin{eqnarray}
  \left( \bigbkr_1^{b-1} A_i \right) \bkr A_b & = & \bigcup_K \left( \left[ \bigbkr_{i=1}^{b-1} A_i \right]_{K} \cap [A_b]_{K^c} \right)  \label{line 1}\\
& = &  \bigcup_K \left( \left[ \bigcup_{J_1, \ldots, J_{b-1}} \bigcap_{i=1}^{b-1} [A_i]_{J_i} \right]_{K}  \cap [A_b]_{K^c} \right)  \label{line 2} \\
& \supseteq & \bigcup_K  \left(  \left( \bigcup_{J_1, \ldots, J_{b-1}} \bigcap_{i=1}^{b-1} [[A_i]_{J_i}]_K \right)  \cap [A_b]_{K^c} \right)   \label{line 3} \\
& = & \bigcup_K \left( \left( \bigcup_{J_1, \ldots, J_{b-1}} \bigcap_{i=1}^{b-1} [A_i]_{J_i\cap K} \right)  \cap [A_b]_{K^c} \right)  \label{line 4} \\
& = & \bigcup_{J_1, \ldots, J_b} \bigcap_{i=1}^b [A_i]_{J_i} = \bigbkr_1^b A_i  \label{line 5}
\end{eqnarray}
The justifications are as follows.  Line \eqref{line 1} is the  
definition, where $K^c$ denotes the complement of $K$.
Line \eqref{line 2} follows by using the definition \eqref{def bkr many}.  
The set inclusion in line \eqref{line 3} results from applying both parts of \eqref{capcup}.
Line \eqref{line 4} follows by applying Lemma \ref{cylcomp} on the composition of cylinder operators.
Line \eqref{line 5} is just re-labeling the indices:  the previous line is a union, indexed by pairwise disjoint 
$J_1, \ldots, J_b$, and a set $K$;  for $i=1$ to $b-1$,  $K_i= J_i \cap K$, and for index $b$, 
we take $K_b = K^c$---the  set of possible indices $\alpha = (J_1 \cap K, \ldots, J_{b-1}\cap K,K^c)$ is identical to the set of  $\alpha=(K_1, \ldots, K_b)$, with $i \ne j$ implies $K_i \cap K_j  = \emptyset$---and then we switch notation back, from $K_i$'s to $J_i$'s.
\end{proof}

\subsection{Probability considerations related to the BKR inequality}\label{sect bkr prob}
References to the BKR inequality were given just after Equation \eqref{cond.2}.
\begin{thm}[The original BKR inequality]\label{bkr thm}
Let $S$ be a finite set, and let $\P$ be a probability measure on $S^d$ for which the $d$ coordinates are mutually independent.  
(The coordinates might have different distributions.)  
For any events $A,B \subseteq S^d$,  with the event $A \bkr B$ as defined by \eqref{def bkr many}, 
\[
    \P(A \bkr B) \le \P(A) \P(B).
\]    
\end{thm}

\begin{cor} \label{ind.bkr} Under the hypotheses of Theorem \ref{bkr thm}, for $ b = 2, 3, \ldots$ and $A_1, \ldots, A_b \subseteq S^d$,
\begin{equation}\label{b inequality}
      \P(A_1 \bkr A_2 \bkr \cdots \bkr A_b) \le \prod_{i=1}^b \P(A_i).
\end{equation}
\end{cor}

\begin{proof}
For $b = 2$, \eqref{b inequality} is the original BKR inequality.  For $b \ge 3$, we apply Proposition \ref{bkr.incl} to see that
\[
\P(A_1 \bkr \cdots \bkr A_b) \le \P((((\cdots ((A_1 \bkr A_2) \bkr A_3) \cdots \bkr A_b).
\]
Applying the $b = 2$ case and induction provides the claim.
\end{proof}

We can now prove Proposition \ref{cond}, which from our new perspective is a simple corollary of the extended BKR inequality, Corollary \ref{ind.bkr}.
\begin{proof}[Proof of Proposition \ref{cond}]
In view of the containment \eqref{contained}, we have:
\[
\P(I \cap W_1 \cap \cdots \cap W_b) \le  \P\left( \bigbkr_1^b W_i \right),
\]
and by Corollary \ref{ind.bkr}
\[
 \P\left( \bigbkr_1^b W_i \right) \le \prod \P(W_i). \qedhere
\]
\end{proof}

\section{The optimization problem we actually solve}

In order to exploit the material in the previous section, 
we \emph{replace} definition \eqref{Pdef.1} of $P$ with
\[
P(\vn; \vec{w}, \vec{p}) := \left(\parbox{3in}{{probability of winning at least $w_i$ times on bet $i$ with $n_i$ tickets, for all $i$,}
and no wins on dependent bets} \right);
\]
from Proposition~\ref{cond}, we know that then
\begin{equation} \label{Peq.2}
   P(\vn; \vec{w}, \vec{p}) \le \prod_{i=1}^b D(n_i; w_i, p_i).
\end{equation}

We will find a lower bound $\vc \cdot \vns$ on the
amount spent by a gambler who did not win dependent bets by solving not \eqref{premother}, but rather
\begin{equation} \label{mother}
   \vc \cdot \vns = \min_{\vn} \vc \cdot \vn\quad \text{s.t.}\quad \ n_i \ge w_i \eand  
    \prod_{i=1}^b D(n_i; w_i, p_i) \ge \eps.
\end{equation}
We furthermore relax the requirement that the numbers of bets, the $n_i$'s, be integers 
and we extend the domain of $D$ to include positive real values of $n_i$ as in \cite[p.~945, 26.5.24]{BarlowProschan}:
\begin{equation} \label{prob}
   D(n; w, p) =  I_p(w, n-w+1), \quad \text{where} \quad
    I_x(a, b) 
    := \frac{\int_0^x t^{a-1} (1-t)^{b-1} \, \dt}{\int_0^1 t^{a-1} (1-t)^{b-1} \, \dt}
\end{equation}
is the 
\emph{regularized Beta function}.  
The function $I_x$, or at least its numerator and denominator, are available in many
scientific computing packages, including Python's SciPy library.
Extending the domain of the optimization problem to non-integral $n_i$ 
can only decrease the lower bound $\vc \cdot \vns$, and it brings two benefits, which we now describe.

In our examples, $\prod D(w_i; w_i, p_i)$ is much less than $\eps$, and consequently $n_i^* > w_i$ for some $i$.  As $D(n; w, p)$ is monotonically increasing in $n$, we have an \emph{equality} $\prod D(n_i^*; w_i, p_i) = \eps$.  This is the first benefit, and it implies by \eqref{Peq.2} an inequality $P(\vns; \vec{w}, \vec{p}) \le \eps$.  Therefore, as in \S\ref{DiMaggio.sec},
if all $N$ people in the gambling population spent at least $\vc \cdot \vns$ on tickets, the probability that one or more of the gamblers would win at least 
$w_i$ times on bet $i$ for all $i$ is at most $N \eps$.  To say it differently: \emph{the solution $\vc \cdot \vns$ to \eqref{mother} is an underestimate of the minimum plausible spending  required to win so many times.}

The second benefit of extending the domain of the optimization problem is to make the problem convex instead of combinatorial.
The convexity allows us to show that any local minimum (as found by the computer) attains the
global minimal value.



\begin{prop} \label{global}
   A local minimizer $\vns$ for the optimization problem \eqref{mother} 
   (relaxed to include non-integer values of $n_i$)
   attains the global minimal value.
\end{prop}

\begin{proof}
We shall show that the set of values of $\vn$ over which we optimize, 
the \emph{feasible set}, 
\begin{equation} \label{global.S}
 \left \{ \vn \in \R^b \mid \text{$n_i \ge w_i$ for all $i$} \right \}
      \cap \left \{ \vn \in \R^b \mid \prod\nolimits_i D(n_i; w_i, p_i) \ge \eps \right\},
\end{equation}
is convex.  
As the \emph{objective function} $\vc \cdot \vn$ is linear in $\vn$, the claim follows.

The first set in \eqref{global.S} defines a polytope, which is clearly convex.
Because the intersection of two convex sets is convex, it remains to show that the
second set is also convex.

The logarithm is a monotonic function, so taking the log of both sides of an
inequality preserves the inequality, and we may write the second set in \eqref{global.S} as:
\begin{equation} \label{global.S2}
  \left \{ \vn \in \R^b \mid \sum\nolimits_i  \log D(n_i; w_i, p_i) \ge \log \eps \right\}.
\end{equation}
For $0 \le x \le 1$ and $\alpha, \beta$ positive, the function
\[
\beta \mapsto \log I_x(\alpha, \beta)
\]
is concave by \cite[Cor.~4.6(iii)]{FinnerRoters}.
Hence $\log D(n_i; w_i, p_i)$ is concave for $n_i \ge w_i$.  
A sum of concave functions is concave, so the set \eqref{global.S2} is a 
convex set, proving the claim.
\end{proof}

\begin{eg}[Louis Johnson 3]
If we solve \eqref{mother} for Louis Johnson's wins---including not only
his Pick~4 wins but also many of his prizes from scratcher games---we find a 
minimum amount spent of at least \$2 million for $\eps = 5 \times 10^{-14}$.
\end{eg}

\subsection*{Monotonicity} 
Some of the gamblers we studied for the investigative report claimed prizes 
in more than 50~different lottery games.  
In such cases it is convenient to solve \eqref{mother} for only a subset 
of the games to ease computation by reducing the number of variables.
Since removing restrictions results in minimizing the same function over a set
that strictly includes the original set, the resulting ``relaxed''
optimization problem still gives a lower bound for the 
gambler's minimum amount spent.


\section{The man from Hollywood} \label{hollywood}

Louis Johnson's astounding 252~prizes is beaten by a man from Hollywood, Florida, whom we refer to as ``H." 
During the same time period, H claimed 570~prizes, more than twice as many as Johnson did.  Yet
Mower's news report \cite{PBP} stimulated a law enforcement action against Johnson but not against H.  
Why?

All but one of H's prizes are in Play~4, which is 
really different from scratcher games: if you buy \$100 worth of scratcher 
tickets for a single \$1 game, this amounts to 100 (almost) independent Bernoulli trials,
each of which is like playing a single \$1 scratcher ticket.  
In Play~4, you can bet any multiple of \$1 on a number to win a given drawing; if you win (which happens with probability $p = 10^{-4}$), 
then you win 5000 times your bet.  
If you bet \$100 on a single Play~4 draw, your odds of winning remain $10^{-4}$, but 
your possible jackpot becomes \$500,000, and if you win, the Florida Lottery records 
this in the list of claimed prizes as if it were 100 separate wins.  Clearly, these are wins on dependent bets.

So, to infer how much H had to spend on the lottery for his wins to be unsurprising, 
first we have to estimate how much he bet on each drawing.  
Unfortunately, we cannot deduce this from the list of claimed prizes, because it 
includes the date the prize was claimed but not the specific drawing the ticket was for.  
(Louis Johnson's Play~4 prizes were all claimed on distinct dates, so it is reasonable 
to assume they were bets on different draws.)
The Palm Beach Post paid the Florida Lottery to retrieve a
sample of H's winning tickets from their archives.
We think H's
winning plays were as in Table~\ref{hollywood.table}.
\newcommand{\tspace}{$\hspace{0.35in}$}
\begin{table}[hbt]
\begin{tabular}{ccr}
date&number played&amount wagered \\ \hline
12/6/2011 & 6251 & \$52$\tspace$ \\
?? & ???? & \$1\tspace \\ 
11/11/2012 & 4077 & \$101\tspace \\
12/31/2012&1195 & \$2\tspace\\
2/4/2013 & 1951 & \$212\tspace \\
3/4/2013 & 1951 & \$200\tspace\\
\end{tabular}
\vspace*{0.7em}
\caption{H's Play~4 wins during 2011--2013} \label{hollywood.table}
\end{table}

To find a lower bound on the amount H spent by solving the optimization problem 
\eqref{mother}, we imagine that he played several different 
Play~4 games, distinguished by their bet size.  
For simplicity, let us pretend that a player can bet 
\$1, \$50, \$100, or \$200, and
suppose we observed H winning these
bets 2, 1, 1, and 2 times, respectively.  
Using these as the parameters in \eqref{mother} and the same probability 
cutoff $\e = 5 \times 10^{-14}$ gives a minimum amount spent of just \$96,354.  

But we can find a number tied more closely to H's circumstances.
In 2011--2013, he claimed \$2.84~million in prizes.  
These are subject to income tax.
If his tax rate was about 35\%, he would have taken home about \$1.85 million.  
If he spent that entire sum on Play~4 tickets, what is the probability that 
he would have won so much?  
We can find this by solving the following optimization problem with 
$p = 10^{-4}$, $\vw = (2, 1, 1, 2)$, and $\vc = (1, 50, 100, 200)$:
\[
   \max_{\vn}  \prod_{i=1}^4 D(n_i; w_i, p) \quad \text{s.t.} \quad \ w_i \le n_i  \eand  \vc \cdot \vn \le 1.85 \times 10^6.
\]
The solution is about $0.0016$, or one-in-625:
it is plausible that H was just lucky.  
That's because he made large, dependent bets, while we know from 
the examples above that betting a similar sum on smaller, 
independent bets is less likely to succeed. 

This illustrates a principle of casino gambling from 
\cite[p.~170]{Dubins} or \cite[\#37]{Mosteller}: \emph{bold play is better than cautious play}.  
If you are willing to risk \$100 betting red-black on a game of roulette, and you only care about doubling your money at the end of the evening, you are better off wagering \$100 on one spin and then stopping, rather than placing 100 \$1 bets.

\section{The real world}

How did this paper come to be?  
One of us, Lawrence Mower, is an investigative reporter in Palm Beach, Florida.
His job is to find interesting news stories and spend 4--6 months investigating them.  
He wondered whether something might be going on with the Florida Lottery, 
so he obtained the list of prizes and contacted the other three of us to help
analyze the data.  
Below we describe some of the non-mathematical aspects.

\subsection*{What some people get up to} 
Various schemes can result in someone claiming many prizes.  

Clerks at lottery retailers have been known to scratch the wax on a ticket lightly with a pin, 
revealing just enough of the barcode underneath to be able to scan it, as described in \cite[paragraph 75]{Ombudsman}. 
If they scan it and it's not a winner, they'll sell it to a customer, who may not 
notice the very faint scratches on the card.  
Lottery operators in many states replaced the linear barcode 
with a 2-dimensional barcode to make this scam more difficult,
but it still goes on: a California clerk was arrested for it on 9/25/14.

Sometimes gamblers will ask a clerk to check whether a ticket is a winner.  
If it is, the clerk might say it's a loser, or might say the ticket
is worth less than it really is, then claim the prize 
at the lottery office---and become the recorded winner.
Of course, most clerks are honest, but this scheme is popular; see, for example, \cite[paragraphs 47, 48, 80, 146]{Ombudsman}.

Another angle, \emph{ticket aggregation}, goes as follows.  
A gambler who wins a prize of \$600 or more may be reluctant to claim the prize at the
lottery office.  
The office might be far away; the gambler might be an illegal alien;
or the gambler might owe child support or back taxes, which the lottery is required to
subtract from the winnings.
In such cases, the gambler might sell the winning ticket to a third party, 
an \emph{aggregator}, who claims the prize and is recorded to be the winner.  
The aggregator pays the gambler less than face value, to cover income tax
(paid by the aggregator) and to provide the aggregator a profit.
The market rate in Florida is \$500-\$600 for a \$1000 ticket.

Some criminals have acted as aggregators to launder money.  
They pay the gambler in cash, but the lottery pays them with a check, ``clean'' money
because it is already in the banking system.  
Notorious Boston mobster Whitey Bulger \cite{Bulger} and Spanish politician 
Carlos Fabra \cite{Fabra} are alleged to have used this dodge.

When questioned by Mower, some of our suspects confessed to aggregating tickets, 
which is a crime in Florida (Florida statute~24.101, paragraph~2).

\subsection*{Outcomes in Florida} 
Before Mower's story appeared, he interviewed Florida Lottery Secretary Cynthia O'Connell about these gamblers.  
She answered that they could be lucky:  
``That's what the lottery is all about.  You can buy one ticket and you become a millionaire'' \cite{PBP}.  
Our calculations show that for most of these 10~gamblers, this is an implausible claim.
O'Connell and the Florida Lottery have since announced reforms 
to curb the activities highlighted here \cite{LotteryResponse}.  
They stopped lottery operations at more than 30~stores across the state
and seized the lottery terminals at those stores.


\subsection*{More news stories and outcomes in other states} 
Further stories about ``too frequent'' winners have now appeared in 
California (KCBS Los Angeles 10/30/14, KPIX San Francisco 10/31/14), 
Georgia (Atlanta Fox 5 News 9/12/14, Atlanta Journal-Constitution 9/18/14), 
Indiana (ABC 6 Indianapolis, 2/19/15), 
Iowa (The Gazette, 1/23/15), 
Kentucky (WLKY, 11/20/14), 
Massachusetts (Boston Globe, 7/20/14), 
Michigan (Lansing State Journal, 11/18/14),  
New Jersey (Asbury Park Press, 12/5/14 \& 2/18/15; USA Today, 2/19/15),
and 
Ohio (Dayton Daily News 9/12/14).
In Massachusetts, ticket aggregation is not illegal per se.  
In California, the lottery makes no effort to track frequent winners.  

In Georgia, ticket aggregation is illegal but the law had not been enforced.
The practice was so widespread  that elementary calculations 
(much simpler than those presented in this article) cast suspicion on 
125~people.  
This gap in enforcement, in principle easy to detect, came 
to light as a consequence of the much more challenging investigation in 
Florida described here.  
This led to a change in policy announced by the Georgia Lottery Director, Debbie Alford, on 9/18/14: 
``We believe that most of these cases involved retailers agreeing to cash winning tickets on behalf of their customers --- a violation of law, rules, and regulations.''

\section*{Acknowledgements} 
We are grateful to Don Ylvisaker, Dmitry B.~Karp, and an anonymous referee for helpful comments and insight.
The second author's research was partially supported by NSF grant DMS-1201542.


\bibliographystyle{amsplain}
\bibliography{lottery}

\providecommand{\bysame}{\leavevmode\hbox to3em{\hrulefill}\thinspace}
\providecommand{\MR}{\relax\ifhmode\unskip\space\fi MR }
\providecommand{\MRhref}[2]{%
  \href{http://www.ams.org/mathscinet-getitem?mr=#1}{#2}
}
\providecommand{\href}[2]{#2}
\begin{thebibliography}{10}

\bibitem{Finding}
Aaron Abrams and Skip Garibaldi, \emph{Finding good bets in the lottery, and
  why you shouldn't take them}, Amer. Math. Monthly \textbf{117} (2010), 3--26.

\bibitem{BarlowProschan}
R.E. Barlow and F.~Proschan, \emph{Statistical theory of reliability and life
  testing}, Holt, Rinehart and Winston, 1975.

\bibitem{CPS}
Jennifer~T. Chayes, Amber~L. Puha, and Ted Sweet, \emph{Independent and
  dependent percolation}, Probability theory and applications ({P}rinceton,
  {NJ}, 1996), IAS/Park City Math. Ser., vol.~6, Amer. Math. Soc., Providence,
  RI, 1999, pp.~49--166.

\bibitem{Bulger}
Kevin Cullen, \emph{{US} orders lottery to hold {B}ulger's winnings}, Boston
  Globe (1995), 1.

\bibitem{DiaconisMosteller}
Persi Diaconis and Frederick Mosteller, \emph{Methods for studying
  coincidences}, J. Amer. Stat. Assoc. \textbf{84} (1989), 853--861.

\bibitem{Dubins}
Lester~E. Dubins and Leonard~J. Savage, \emph{How to gamble if you must:
  inequalities for stochastic processes}, McGraw-Hill, 1965.

\bibitem{Durrett:EOSP}
Rick Durrett, \emph{Essentials of stochastic processes}, 2nd ed., Springer,
  2013.

\bibitem{Ethier}
S.N. Ethier and Davar Khoshnevisan, \emph{Bounds on gambler's ruin
  probabilities in terms of moments}, Methodology and Computing in Applied
  Probability \textbf{4} (2002), 55--68.

\bibitem{Fabra}
Mar{\'\i}a Fabra, \emph{The law finally catches up with former {C}astell{\'o}n
  cacique {F}abra}, El Pa\'is in English (2013).

\bibitem{FinnerRoters}
H.~Finner and M.~Roters, \emph{Log-concavity and inequalities for chi-square,
  {F} and {B}eta distributions with applications in multiple comparisons},
  Statistica Sinica \textbf{7} (1997), 771--787.

\bibitem{ProbTales}
Charles~M. Grinstead, William~P. Peterson, and J.~Laurie Snell,
  \emph{Probability tales}, Student Mathematical Library, vol.~57, American
  Mathematical Society, 2011.

\bibitem{GroteMatheson}
Kent~R. Grote and Victor~A. Matheson, \emph{In search of a fair bet in the
  lottery}, Eastern Economic Journal \textbf{32} (2006), 673--684.

\bibitem{Milky}
Paul~W. Hodge, \emph{Galaxies}, Harvard University Press, 1986.

\bibitem{Numbers}
Lawrence~J. Kaplan and James~M. Maher, \emph{The economics of the numbers
  game}, American Journal of Economics and Sociology \textbf{29} (1970), no.~4,
  391--408.

\bibitem{Ombudsman}
Andr\'e Marin, \emph{A game of trust: investigation into the {O}ntario
  {L}ottery and {G}aming {C}oproration's protection of the public from fraud
  and theft}, Ombudsman Report, March 2007.

\bibitem{Mosteller}
F.~Mosteller, \emph{Fifty challenging problems in probability with solutions},
  Dover, 1987, reprint of the 1965 Addision-Wesley edition.

\bibitem{PBP}
Lawrence Mower, \emph{Gaming the lottery}, The Palm Beach Post, March 30 2014.

\bibitem{LotteryResponse}
Cynthia O'Connell, \emph{Point of view: lottery working to improve security,
  fight fraud}, Palm Beach Post (2014), Opinion page.

\bibitem{Reimer}
David Reimer, \emph{Proof of the van den {B}erg-{K}esten {C}onjecture},
  Combinatorics, Probability and Computing \textbf{9} (2000), no.~1, 27--32.

\bibitem{Sellin}
Thorsten Sellin, \emph{Organized crime: a business enterprise}, Annals of the
  {A}merican {A}cademy of {P}olitical and {S}ocial {S}cience \textbf{347}
  (1963), 12--19.

\bibitem{Stefanski}
Leonard~A. Stefanski, \emph{The {N}orth {C}arolina lottery coincidence}, The
  American Statistician \textbf{62} (2008), no.~2, 130--134.

\bibitem{BFiebig}
J.~van~den Bergh and U.~Fiebig, \emph{On a combinatorial conjecture concerning
  disjoint occurrences of events}, Annals of Probability \textbf{15} (1987),
  no.~1, 354--374.

\bibitem{BK}
J.~van~den Bergh and H.~Kesten, \emph{Inequalities with applications to
  percolation and reliability}, J. Appl. Prob. \textbf{22} (1985), 556--569.

\bibitem{Zillow}
Zillow.com, \emph{Florida home prices and home values}, January 22 2014.

\end{thebibliography}

\end{document}